\documentclass[10pt]{amsart}
\usepackage[all]{xy} 
\usepackage{verbatim}
\usepackage{tikz}
\usepackage{amsfonts, amssymb}

\newtheorem{thm}{Theorem}[section]
\newtheorem{cor}[thm]{Corollary}
\newtheorem{prop}[thm]{Proposition}
\newtheorem{lem}[thm]{Lemma}
\theoremstyle{definition}
\newtheorem{defn}[thm]{Definition}
\newtheorem{ex}[thm]{Example}
\newtheorem{rmk}[thm]{Remark}
\newtheorem*{ack}{Acknowledgments}

\numberwithin{equation}{section}

\newcommand{\U}{U}

\newcommand{\Kalt}{\hat{L}^{-1}}
\newcommand{\aalt}{b}
\newcommand{\Ialt}{J}
\newcommand{\Ralt}{S}

\newcommand{\Tr}{\mathrm{Tr}}

\newcommand{\sgn}{{\text{sgn}}}

\newcommand{\R}{\mathbb R}
\newcommand{\C}{\mathbb C}
\newcommand{\Z}{\mathbb {Z}}
\newcommand{\N}{\mathbb {N}}
\newcommand{\Omb}{\Omega}

\newcommand{\Om}{\Omega}

\title{An elementary differential extension of odd K-theory}

\author[T.~Tradler]{Thomas~Tradler}
  \address{Thomas Tradler,
  Department of Mathematics, College of Technology, City University of New York, 300 Jay Street, Brooklyn, NY 11201}
  \email{ttradler@citytech.cuny.edu}

\author[S.~Wilson]{Scott O. Wilson}
  \address{Scott O. Wilson, Department of Mathematics, Queens College, City University of New York, 65-30 Kissena Blvd., Flushing, NY 11367}
  \email{scott.wilson@qc.cuny.edu}

\author[M.~Zeinalian]{Mahmoud~Zeinalian}
  \address{Mahmoud Zeinalian, Department of Mathematics, C.W. Post Campus of Long Island University, 720 Northern Boulevard, Brookville, NY 11548, USA} 
  \email{mzeinalian@liu.edu}

\begin{document}
\maketitle

\begin{abstract}
There is an equivalence relation on the set of smooth maps of a manifold into the stable unitary group, defined using a Chern-Simons type form, whose equivalence classes form an abelian group under ordinary block sum of matrices. 
This construction is functorial, and defines a differential extension of odd K-theory, fitting into natural commutative diagrams and exact sequences involving K-theory and differential forms. To prove this we obtain along the way several results concerning even and odd Chern and Chern-Simons forms. 
\end{abstract}
\allowdisplaybreaks

\setcounter{tocdepth}{2}
\tableofcontents

\section{Introduction}

Differential cohomology theories provide a geometric refinement of cohomology theories which, intuitively, contain additional differential-geometric information. An important first example is \emph{ordinary differential cohomology}
where elements of ordinary cohomology, such as a first Chern class, are enriched by additional geometric data, such as a line bundle with connection, \cite{ChS}, \cite{D}.

The existence of such a differential refinement of a cohomology theory was proposed and proved by Hopkins and Singer in \cite{HS}, with an axiomatic characterization given in \cite{BS3}.  Several recent works have centered around the construction and 
properties of differential K-theory,  \cite{HS},  \cite{BS},  \cite{BS3}, \cite{L}, \cite{SS}, \cite{FL}, whose additive structure splits into even and odd degree parts, just as in K-theory. 

As with all cohomology and differential theories, it is particularly important to have nice geometric models for the theories, as these provide both a means by which to understand the theories, and the form in which they often appear in mathematical and physical discussions.  

In  \cite{SS} the authors construct the even degree part of differential K-theory geometrically 
using bundles with connection up to Chern-Simons equivalence.  In particular, they show the data of an extra odd differential form appearing in several other presentations is superfluous.  In this paper we provide the analogous geometric model for a differential extension of the odd degree of K-theory. 

For the model given here, a cocycle is simply a smooth map into the infinite unitary group, where equivalence of cocycles is 
 determined by a Chern-Simons type differential form, \emph{i.e.} the transgression form for the odd Chern character form. 
 Several new results proved here show that the resulting set of equivalence classes form an abelian group, determining a differential extension of odd K-theory (Definition \ref{defn:diffext} and Theorem \ref{thm:diffext}). 

Of particular importance, we show every exact odd form on any manifold $M$ is obtained as the odd Chern form of some map into the stable unitary group (Corollary \ref{cor:ChSurjOnExact}), and every even differential form, modulo exact, is given as a Chern-Simons form of a path of maps into the stable unitary group (Corollary \ref{cor:CSsurj}).
Finally, we also prove a differential-form version of Bott periodicity, identifying all even Chern forms, modulo exact forms, with the set of  Chern-Simons forms obtained from based loops in $Map(M,U)$ (Theorem \ref{thm:CSCh}).

In the appendix we show for completeness sake that if one includes the data of even differential forms 
into the construction, then the resulting differential extension (which is closer to the models in the literature) is in fact isomorphic to the leaner model in the bulk of the paper. 
The advantage of our first model lies in that it does not require the data of an additional even differential form.
In particular, this implies that a map from a compact manifold $M$ into the infinite unitary group $U$ contains calculable information beyond the homotopy class representing an element in $K^{-1}(M)$, and also calculable information beyond the odd Chern character given by pullback of the canonical odd form on $U$.

In \cite{BS} it is shown that differential extensions of odd K-theory are not unique. In fact there are ``exotic additive structures'' one can put on the underlying set of odd differential K-theory which still yield an odd differential extension. 
Nevertheless, since the abelian group structure we construct here is induced by ordinary block sum of matrices, it seems quite likely that the differential extension given here agrees with that of odd differential K-theory, which itself is unique up to a unique natural isomorphism \cite{BS3}.  

In future work we hope to build a multiplicative structure for the differential extension given here, which would in particular imply that the constructions in this paper yield odd differential K-theory, \cite{BS3}.
We also hope to apply the ideas in this paper to yield an elementary differential refinement of algebraic K-theory.
Finally, just as in \cite{TWZ} where it is shown that the even part of differential K-theory admits a refinement 
incorporating the free loopspace, we expect the same to be true for the odd version constructed here, and for this to have interesting applications to loop groups. 

\begin{ack}
We would like to thank Dennis Sullivan, Dan Freed, and Ulrich Bunke for useful conversations concerning the topics of this paper. The first and second authors were supported in part by grants from The City University of New York PSC-CUNY Research Award Program.
\end{ack}

\section{The odd Chern form}\label{SEC:odd-chern-form}
 
In this section we give elementary definitions of the odd Chern forms in differential geometry. We will consider the category of compact manifolds $M$ with corners, and smooth maps between them. Let $\Om^{\textrm{even}}(M;\R)$ and $\Om^{\textrm{odd}}(M;\R)$ denote the space of even and odd real differential forms on $M$, and 
$\Om^{\textrm{even}}_{cl}(M;\R)$ and $\Om^{\textrm{odd}}_{cl}(M;\R)$ denote the subspaces of closed differential forms.

Recall that for a complex bundle $E\to M$ with hermitian metric, equipped with a unitary connection $\nabla$, and associated curvature $R$, we have the following closed even real valued differential form, known as the Chern character\footnote{We note that the factor $1/(2 \pi i)$ in \eqref{eq:Ch(nabla)} is necessary and used in Theorem \ref{thm:CSCh}.},
\begin{equation}\label{eq:Ch(nabla)}
Ch(\nabla) = \Tr \left( e^{\frac{1}{2 \pi i} R}\right) \in \Om^{\textrm{even}}_{cl}(M;\R).
\end{equation}
Furthermore, for a path of such connections $\nabla^s$, there is an odd differential form
$CS(\nabla^s) \in \Om^{\textrm{odd}}(M;\R)$, known as the 
\emph{Chern-Simons form},
\begin{multline} \label{eq:CS}
CS(\nabla^s)\\
= \Tr \left( \int_0^1 \sum_{n\geq 0} \frac{1}{(2 \pi i)^{n+1}} \frac{1}{(n+1)!} 
\sum_{i=1}^{n+1} 
\overbrace{ R^s \wedge \dots\wedge R^s\wedge \underbrace{ \left( \nabla^s \right)' }_{i^{\text{th}}}\wedge R^s\wedge\dots\wedge R^s }^{n+1\text{ factors}} \right) ds 
\end{multline}
where $\left( \nabla^s \right) ' = \frac{\partial}{\partial s}  \nabla^s $ and $R^s$ is the curvature of $\nabla^s$. This form satisfies the property 
\begin{equation}\label{eq:CS(nabla^s)}
d CS(\nabla^s) = Ch(\nabla^1) - Ch(\nabla^0).
\end{equation}

Let  $U = \lim\limits_{n \to \infty}U(n)$ be the stable unitary group. Elements in this group can be interpreted as infinite unitary matrices whose entries differ from the identity in only finitely many places. So, such an element determines an element of $U(N)$ for $N$ sufficiently large, which is well defined up to block sum with an identity matrix.

 For any compact smooth manifold with corners $M$, let $Map(M, \U )$ be the topological space of smooth maps from $M$ to $U$, i.e.
 $Map(M, \U ) =  \lim\limits_{n \to \infty} Map(M,U(n))$, where $Map(M,U(n))$ is given the compact open topology.
  This is a topological monoid under the ordinary product of matrices in $U$.  It also true that $Map(M, \U )$ is a monoid under block sum, denoted $\oplus$, though this operation is not continuous. On homotopy classes of maps, which defines the set we denote by $K^{-1}(M)$, these two products are equal, and define an abelian topological group.

An element $g \in Map(M, \U )$ determines a connection $d + g^{-1} dg$ on the trivial bundle  
$\C^N \times M \to M$ for some large $N$, which is unitary with respect to the standard hermitian metric. This bundle with connection is well defined up to  direct sum with another trivial bundle $\C^k \times M \to M$, with standard flat ``zero'' connection $d$. Note that $d + g^{-1} dg$ is gauge equivalent to $d$ and flat, and so  $Ch(d + g^{-1} dg) = Ch(d)$ .
So,  for all $g \in Map(M, \U )$, if  $\nabla^s = d + s g^{-1} dg$ then by Equation \eqref{eq:CS(nabla^s)}, 
$CS(\nabla^s) \in \Omega^{\textrm{odd}}(M)$ is $d$-closed since
\[d\big(CS(d + sg^{-1} dg)\big)= Ch(d + g^{-1} dg) - Ch(d) = 0.\] 
 \begin{defn} 
 We define the odd Chern Character map, 
\[
Ch: Map(M, \U ) \to \Omega^{\textrm{odd}}_{cl} (M;\R),
\]
 by 
\begin{equation}\label{eq:Ch(g)}
Ch(g) = CS(d + s g^{-1} dg).
\end{equation}
\end{defn}

We remark that the definition above is independent of the integer $N$ chosen to realize the bundle with connection over $M$,
since the addition of an identity block to $g$ changes $g^{-1} dg$ only by a block sum with a zero matrix. Additionally,
this odd form is independent, modulo an exact form, of the choice of trivial connection $d$, see \cite{Zh}.
The notation is chosen so as to mirror exactly the presentation of the Chern-Character 
and Chern-Simons form in the even case, with a degree shift by one. 

The odd Chern character can be computed explicitly, see for example \cite{G}.

\begin{lem}\label{LEM:Ch-G=Ch-TWZ}
For any $g \in Map(M, \U )$ we have
\[
Ch(g) =  \Tr \sum_{n\geq 0} \frac{(-1)^n }{(2 \pi i)^{n+1}} \frac{n!}{(2n+1)!}   (g^{-1}dg)^{2n+1} 
\]
\end{lem}

\begin{proof} Let $A = g^{-1}dg$  and $A_s = s g^{-1}dg$ so that $(A^s)'=A$ and $R^s=-s(1-s)A\cdot A$. Then 
\begin{eqnarray*}
Ch(g) &=& \Tr  \sum_{n\geq 0} \frac 1 {(n+1)!} \frac{1}{(2 \pi i)^{n+1}} \sum_{i=1}^{n+1} \int_0^1 R^s \dots \underbrace{(A^s)'}_{i\text{-th}}\dots R^s ds   \\
&=&   \sum_{n\geq 0} \frac 1 {n!} \frac{1}{(2 \pi i)^{n+1}} \Tr \int_0^1 (-s(1-s))^n \underbrace{A\dots A}_{2n+1 \text{ factors}} ds \\
&=& \sum_{n\geq 0} \frac{1}{n!}\frac{1}{(2 \pi i)^{n+1}} (-1)^n \frac{n!n!}{(2n+1)!}  \Tr(A^{2n+1})  \\
&=& \sum_{n\geq 0} \frac{(-1)^n}{(2 \pi i)^{n+1}} \frac{n!}{(2n+1)!} \Tr(A^{2n+1}) 
\end{eqnarray*}
where the relation $\int_0^1  s^k (1-s)^\ell ds=\frac{k!\ell!}{(k+\ell+1)!}$ was used.
\end{proof}

\begin{defn}\label{def:Theta}
Let $\omega$ be the canonical left-invariant $1$-form on $U(n)$, with values in the Lie algebra $\mathfrak{u} (n)$ of $U(n)$,
and define
 $\Theta \in \Om^{\textrm{odd}}( U(n);\R)$ by
\begin{equation}\label{eq:Theta}
\Theta = \Tr \sum_{n\geq 0} \frac{(-1)^n }{(2 \pi i)^{n+1}} \frac{n!}{(2n+1)!}   \omega^{2n+1} .
\end{equation}
Note this is real valued since $\omega / i$ and $\omega^2/i $ are hermitian.
\end{defn}

The odd Chern character $Ch(g)$ is simply the pullback of $\Theta$ along a map $g: M \to U$. 
\begin{lem}\label{lem:Ch-from-Theta}
For any $g: M \to U(n)$, we have 
\[Ch(g) = g^*(\Theta)=\Tr \sum_{n\geq 0} \frac{(-1)^n }{(2 \pi i)^{n+1}} \frac{n!}{(2n+1)!}   g^*(\omega)^{2n+1}.\]
\end{lem}
\begin{proof}
This follows from Lemma \ref{LEM:Ch-G=Ch-TWZ} and the fact that $g^*(\omega)=g^{-1}dg$.
\end{proof}

It follows that the odd Chern character is natural via pullback along maps $f: N \to M$, and that $Ch(g)$ is a well defined piecwise smooth form whenever $g: M \to \U$ is piecewise smooth. 

For two elements $g,h\in Map(M, \U )$, we denote by $g\oplus h\in Map(M, \U )$ and by $g^{-1}\in Map(M, \U )$ the elements given by taking the block sums and inverses in $U$, respectively.

\begin{cor} \label{cor:Chadd}
The Chern Character map $Ch: Map(M, \ U) \to \Omega^{\textrm{odd}}_{cl} (M;\R)$
is a monoid homomorphism,
i.e.
\[ 
Ch(  g \oplus h ) = Ch(g) + Ch(h),
\]
and furthermore satisfies $Ch(g^{-1}) = -Ch(g)$.
\end{cor}

\begin{proof} These follows from the previous Lemma \ref{LEM:Ch-G=Ch-TWZ} since
\[
( g \oplus h)^{-1}d(g \oplus h) = (g^{-1}dg) \oplus h^{-1}dh,
\]
trace is additive, and cyclically invariant, and $(g^{-1})^{-1} d(g^{-1}) = -dg g^{-1}$.
\end{proof}

\begin{ex}[$M$ is a sphere] \label{ex:oddChSphere}
For all $n \geq 0$, there is a map $g: S^{2n+1} \to \U$ such that
the degree $2n+1$ part of the differential form $Ch(g)$ equals a non-zero constant multiple of the standard volume form 
on the $(2n+1)$-sphere $S^{2n+1} \subset \R^{2n+2}$. The earliest reference we found for this
 is \cite[p. 1496, eq. (2.13)]{LT}, where the authors construct a map $g : S^{2n+1}  \to U(2^{n+1})$ 
 by use of the Clifford algebra $Cl_{-}(2n+2)$ with generators $\gamma_1, \ldots, \gamma_{2n+3}$. 
 Explicitly, the map is given by
\[
g = \gamma_{2n+2} \sum_{i=1}^{2n+2} \gamma_i x_i
\]
for $(x_1, \ldots, x_{2n+2}) \in S^{2n+1} \subset \R^{2n+2}$. We remark that this construction satisfies two further properties: 
$Ch(g)$ is exact in degrees $k < 2n+1$ and vanishes in degrees $k > 2n+1$. These follow from the deRham theorem and the cohomology of $S^{2n+1}$, since $Ch(g)$ is closed.
\end{ex}

\begin{cor} \label{cor:ChSurjOnExact}
Let $Map^o(M, \U )$ denote the identity component of \linebreak $Map(M, \U )$. For all compact manifolds $M$ with corners,
the image of the map $Ch: Map^o(M, \U ) \to  \Omega^{\textrm{odd}}_{cl} (M;\R)$ contains 
all exact odd real valued differential forms on $M$.
\end{cor}

 \begin{rmk}
 It is already known from \cite{SS} that every odd form on $M$ equals $CS(\nabla^s)$ for some path of connections $\nabla^s$.   
 Corollary \ref{cor:ChSurjOnExact} states furthermore that, if the odd form on $M$ is \emph{exact},
 the path of connections can be chosen to be $\nabla^s = d + s g^{-1} dg$  for some $g \in Map^o(M, U)$.
 \end{rmk}

\begin{proof} If $\omega \in \Om^1(M;\R)$ is exact,
 then $\omega = df$ for some function $f: M \to \R$. We let $g = e^{2 \pi i f}$,
so that $Ch(g)$ equals $df= \omega$, where $Ch(g)$ vanishes outside degree $1$ since $(g^{-1} dg)^2 = 0$. Finally, 
note $g \in Map^o(M, U(1) )$ since $g_t = e^{2 \pi i t\cdot f }$ is a path to the constant identity map.

By induction we assume we may obtain all odd exact forms on $M$ of degree less than or equal to $2j-1$ as $Ch(g)$ 
for some map $g \in Map^o(M, \U )$. We may always choose an imbedding $M \to \R^k$ for some $k > 2j + 1$, and since the pullback map is natural and surjective on exact forms, and $Ch$ is natural via pullback, it suffices to show 
for every exact $\omega \in \Om^{2j+1}(\R^k)$ there is a map $Map^o(\R^k, \U )$ such that $Ch(g)$ equals $\omega$ in degree
$2j+1$, vanishes in higher degree, and is exact in lower degrees. Then by the inductive hypothesis, using the relations
in Corollary \ref{cor:Chadd}, we may use block sum and inverses to construct
 all odd exact forms on $M$ of degree less than or equal to $2j+1$ as $Ch(g)$ for some 
 $Map^o(\R^k, \U )$, completing the induction.
 
An arbitrary exact $2j+1$ form on $\R^k$ can be written as a sum of terms of the form $df dx_{i_1} \ldots dx_{i_{2j}}$, so again by Corollary \ref{cor:Chadd}, using block sums, it suffices to show that for an arbitrary function $f: \R^k \to \R$  there is a map $h: W \to \U$, where $W \subset \R^k$ is some neighborhood of $M$, such that $Ch(h)$ equals $df dx_{1} \ldots dx_{2j}$ in degree $2j+1$, vanishes in higher degree, and is exact in lower degrees. In order to define $h: W \to \U $, we consider the following composition of maps,
\[
\xymatrix{
  W \ar[r]^-F \ar[r]  & D_r^{2j+1} \ar[r]^-s & S^{2j+1} \ar[r]^-g & U(2^{j+1}) \subset \U \\
 \frac 1 N df dx_{1} \ldots dx_{2j} &  dx_{1} \ldots dx_{2j+1}=s^*(dV) & dV &
}
\]
Here $g: S^{2j+1} \to U(2^{j+1}) \subset \U$ is from Example \ref{ex:oddChSphere}, giving up to a constant the volume form  $dV$  on $S^{2j+1}$ as the degree $2j+1$ component of $Ch(g)$, with all lower degree components exact, and all higher degree components vanishing. Let $s: D^{2j+1} \to S^{2j+1}$ be any smooth imbedding of the disk into the sphere. 
 The pullback of the form $dV$ on the sphere is some volume form on the disk. Using Moser's theorem, and rescaling the disk to appropriate radius $r$, we may assume that this pullback form is the standard volume form $dx_1 \cdots dx_{2j+1}$ on  $D_r^{2j+1} $.
 
 Now, for some contractible neighborhood $W\subset \R^k$ contained in a compact set, choose $N \in \N$ so that 
the map    
\[
F(x_1, \ldots , x_k) = (1/N f(x_1, \ldots , x_k), x_1, \ldots , x_{2j})
\]
has image contained in $D^{2j+1}$. Then $F^*(dx_{1} \ldots dx_{2j+1} ) =  1/ N df dx_{1} \ldots dx_{2j}$, so that the pullback gives $1/N$ times the desired form $dfdx_{i_1}\dots dx_{i_{2j}}$. Taking the $N$-fold block sum of this map $(g\circ s\circ F):W\to \U $ with itself defines the desired function $h:  W \to \U$. Finally, we have $h \in Map^o(W, \U ) $ since $W$ is contractible.
\end{proof}

\section{The Chern Simons form}

 It is a natural question to ask how the odd Chern form $Ch(g)$ depends on $g$. Let $I = [0,1]$. For a smooth map $g_t : M \times I \to \U(n) \subset \U$, there is a smooth even differential form that interpolates between $Ch(g_1)$ and $Ch(g_0)$, in the following way. 
We define
 $CS(g_t) \in \Om^{\textrm{even}}(M;\R)$ by
 \begin{equation}\label{eq:CS(gt)}
 CS(g_t) = \int_I Ch(g_t),
\end{equation}
where $Ch(g_t) \in  \Om^{\textrm{odd}}_{cl}(M\times I;\R)$ and $\int_I$ is the integration along the fiber
\[
\xymatrix{
M \times I \ar[r]^-{g_t} \ar[d]_-{\int_I} & U\\
M & 
}
\]
We note that $CS(g_t)$ is independent of parameterization, and so any $g_t: M \times [a,b] \to \U$ can always be reparameterized to domain $M \times I$, without changing the Chern-Simons form on $M$.

By Stokes' theorem we have
\begin{equation*}
d CS(g_t) = d\int_I Ch(g_t) = \int_{\partial I} Ch(g_t) - \int_I d Ch(g_t) =Ch(g_1) - Ch(g_0)
\end{equation*}
since $d Ch(g_t) = 0$.
 
 \begin{lem} \label{oddCS}
 For any $g_t \in Map(M \times I, \U )$ the \emph{odd Chern Simons form} 
 $CS(g_t)  \in  \Omega^{\textrm{even}} (M;\R)$ associated to $g_t$  is  
\[
CS(g_t) = \Tr \sum_{n\geq 0} \frac{(-1)^n}{(2 \pi i)^{n+1}}  \frac{n!}{(2n)!} \int_0^1 (g_t^{-1} g_t') \cdot (g_t^{-1} dg_t)^{2n} dt .
 \]
\end{lem}
 
 \begin{proof}
By Lemma \ref{lem:Ch-from-Theta}, we have $Ch(g_t)=\Tr \sum_{n\geq 0} \frac{(-1)^n }{(2 \pi i)^{n+1}} \frac{n!}{(2n+1)!}  g_t^*(\omega)^{2n+1}$.  The pullback of the canonical left invariant form $\omega$ on $U(n)$ along the map $g_t: M \times I \to U(n)$ is $g_t^*(\omega) = \iota_{\frac{\partial}{\partial t}} (g_t^{-1} dg_t)  dt  + g_t^{-1}dg_t = (g_t^{-1} g'_t ) dt  + g_t^{-1}dg_t$, so the result follows by the definition of integration along the fiber, and the fact that trace is cyclic.
 \end{proof}

We restate the fundamental property for $CS(g_t)$ here, along with several others,  whose proofs are immediate from the definitions and Lemma \ref{oddCS}.
\begin{prop} \label{prop:CSprop}
For any paths $g_t, h_t \in Map(M\times I, \U )$  we have
\begin{eqnarray*}
  d CS(g_t) &=& Ch(g_1) - Ch(g_0) \\
  CS(g_t \oplus h_t) &=& CS(g_t) + CS (h_t) \\
  CS(g_t^{-1}) &=& - CS(g_t)
\end{eqnarray*}
If $g_t$ and $h_t$ can be composed (\emph{i.e.} if $g_1=h_0$), then the composition $g_t * h_t$ satisfies
\[ CS(g_t*h_t)=CS(g_t)+CS(h_t). \]
\end{prop} 

We note that composition may be done smoothly by defining $g_t * h_t$ to be constant equal to $g_1=h_0$ on some interval, and reparametizing to obtain a function on $M \times I$.

The following lemma shows that the degree zero part of the even form $CS(g_t)$ can be interpreted as a ``winding number.''
Let $\Omb U$ be the space of smooth based loops in $U$.

\begin{lem} For any  based loop $g_t \in \Omb U$, 
\[
\frac { 1 }{ 2 \pi i } \Tr \int_0^1 g_t^{-1} g_t' \,  dt \in \Z
\]
\end{lem}

\begin{proof} Choose $n$ such that $g_t \in \Omb U(n)$. We have 
\[
 SU(n)  \rtimes U(1) \cong  U(n) 
 \]
Under the isomorphism $(n, h) \mapsto nh$, the semi-direct group structure is given by 
\[
(n_1, h_1) \cdot (n_2, h_2) = (n_1 h_1n_2 h_1^{-1} , h_1h_2) \quad \quad (n , h)^{-1} = (h^{-1}n^{-1}h,h^{-1})
\]
Therefore, if $g_t = (n_t , h_t)$ under the isomorphism above, we have
\[
g_t^{-1} g_t' = h_t^{-1} n_t^{-1} n_t' h_t +  h_t^{-1} h_t' 
\]
where $g_t^{-1} g_t' \in \mathfrak{u}(n)$ and $ n_t^{-1} n_t' \in \mathfrak{su}(n)$ and   
$h_t^{-1} h_t' \in \mathfrak{u}(1) = i \R$. Since $\mathfrak{su}(n)$ consists of trace zero matrices, 
$\Tr (h_t^{-1} n_t^{-1} n_t' h_t) = 0$, so 
it's enough to show 
\[
\frac { 1 }{ 2 \pi i } \int_0^1 h_t^{-1} h_t' dt \in \Z.
\]
For any path $h_t : [0,1] \to  U(1)$ satisfying $h(0) = 1$ we have $h(s) = e^{\int_0^s h_t^{-1} h_t' dt}$ 
(\emph{i.e.} the integral $\int_0^s h_t^{-1} h_t' dt$ is the logarithmic lift of $h_t$ by the covering exponential map), since
\[
k(s) = h_s^{-1} e^{\int_0^s h_t^{-1} h_t' dt}
\]
satisfies $k(0) = 1$ and $k'(s) = 0$ for all $s$. So, for $h(t) \in \Omb U(1)$, 
\[
e^{\int_0^1 h_t^{-1} h_t' dt }= h(1)=1, \quad\text{so that } \int_0^1 h_t^{-1} h_t' dt \in 2 \pi i \Z.
\]
\end{proof}

If we move $g_t$ through a smooth family $g_t^s$ with fixed endpoints, the form $CS(g^s_t)$ changes only by an exact form, as we now show.
For $g_t^s : M \times I \times I \to U$, where $(s,t) \in I \times I$, let $H(g_t^s)\in \Om^{\textrm{odd}}(M;\R)$ be given by
\[
H(g_t^s) = \int_{I \times I} Ch(g_t^s) 
\]
where  $Ch(g_t^s) \in \Om^{\textrm{odd}}(M \times I \times I;\R)$ and $\int_{I\times I}$ is the integration along the fiber
\[
\xymatrix{
M \times I \times I \ar[r]^-{g_t^s } \ar[d]_-{\int_{I\times I}} & U\\
M & 
}
\]
By Stokes' theorem we have
\[
d H(g_t^s) =  \int_{\partial (I \times I) } Ch(g_t) + \int_{I\times I} d Ch(g_t) = 
CS(g_t^1) - CS(g_t^0) - CS(g_1^s) + CS(g_0^s)
\]
since $d Ch(g_t) = 0$.  In particular we have

\begin{prop} \label{prop:CSuptoExact}
Let $g_t^s : M \times I \times I \to U$ be a smooth map which is constant along $t=0$ and $t=1$, i.e. 
$g_0^s = g_0^0$ and $g_1^s= g_1^0$ for all $s$. Then 
\begin{multline*}
H(g_t^s) = \Tr \sum_{\stackrel{n \geq 0}{\stackrel{1 \leq i,j \leq 2n+1}{i \neq j}}}
\frac{(-1)^n }{(2 \pi i)^{n+1}} \frac{n!}{(2n+1)!}  \\
\cdot \int_0^1\int_0^1 \overbrace{
 g^{-1}dg \cdots  \underbrace{(g^{-1} \frac{\partial}{\partial t} g )}_{i^{\text{th}}} 
\cdots  \underbrace{ (g^{-1}  \frac{\partial}{\partial s} g) }_{j^{\text{th}}}   \cdots g^{-1}dg }^{\text{$2n+1$ terms}} dtds
\end{multline*}
where $g = g_t^s$, and satisfies
\[
dH(g_t^s)  = CS(g_t^1) - CS(g_t^0).
\]
\end{prop}
\begin{proof}
The second claim is a special case of what was shown above. The first follows from Lemma \ref{lem:Ch-from-Theta}, the formula for integration along the fiber, and the fact that $g^*(\omega) = (g^{-1} \frac{\partial}{\partial t} g )dt +  (g^{-1} \frac{\partial}{\partial s} g )ds + (g^{-1} dg )$, where $g = g_t^s$, since $\omega$ is a $1$-form.
\end{proof}

The following theorem shows that the collection of even forms $\{CS(g_t)\}$,  for $g_t: M \times S^1 \to U$ satisfying $g_0=g_1=1$, are the same modulo exact forms, as the set of Chern forms $\{Ch(\nabla)\}$ for some connection $\nabla$.
As explained in Remark \ref{rmk:CSCh}, this is a version of Bott periodicity concerning closed differential forms.

\begin{thm} \label{thm:CSCh}
Let $M$ be a closed manifold. For every bundle with connection $(E,\nabla)$, there is a map $g : M \to \Omb U$ such that $CS(g_t) = Ch(\nabla)$. Conversely, for every map $g : M\to\Omb U$ there is a bundle with connection $(E,\nabla)$ such that $Ch(\nabla) \equiv CS(g_t)$  \emph{modulo exact forms}.
\end{thm}

\begin{proof}
Let $(E,\nabla)$ be a bundle with connection over $M$, whose Chern character is $Ch(\nabla)$. By a theorem of Narasimhan and Ramanan \cite{NR}
there is an isomorphism of $E$ onto a subbundle of a trivial $\C^N$-bundle, such that 
$\nabla$ is obtained as the restriction of the trivial connection $d$ on this trivial $\C^N$-bundle. 
Let $P: M \to \mathcal{M}_{N \times N}$ be the projection operator whose image defines this bundle, so that 
the connection on $Im(P)$ is given by $\nabla (s) = P( ds)$. A calculation shows the 
curvature of $\nabla$ is given  by $R(s) = P(dP)^2(s)$, so that
\[ 
Ch(\nabla)=\Tr (e^{\frac{1}{2 \pi i} R})=\Tr \sum_{n\geq 0} \frac{1}{(2 \pi i)^{n}} \frac{1}{n!} P(dP)^{2n}
\]

We define $g_t: M \times S^1 \to U$ by 
\[
g_t = e^{ 2 \pi i P t} = Id  + (e^{2 \pi i t} - 1) P 
\]
and show that $CS(g_t) = Ch(\nabla)$. We have
\[
g_t^{-1} g'_t  = 2 \pi i P
\]
and
\begin{eqnarray*}
g_t^{-1} d g_t &=& \left( Id + (e^{-2 \pi i t} - 1) P \right) \left( ( e^{2 \pi i t} - 1) dP \right)
\\
&=& ( e^{2 \pi i t} - 1) dP +( e^{2 \pi i t} - 1)( e^{-2 \pi i t} - 1) P dP. 
\end{eqnarray*}
Since $P^2=P$, we have $dP P+PdP=dP$, so $PdP=dP-dP P=dP P^\perp$, so $(P dP )^2= 0$ and so $d P P dP = P^\perp (dP)^2$. Therefore, 
\begin{eqnarray*}
(g_t^{-1} d g_t )^2 &=& ( e^{2 \pi i t} - 1)^2 (dP)^2 + ( e^{2 \pi i t} - 1)^2(e^{-2 \pi i t} - 1)\left( P (dP)^2 + P^\perp (dP)^2 \right) \\
&=& ( e^{2 \pi i t} - 1)^2 e^{-2\pi i t} (dP)^2 =-4 \sin^2 ( \pi t) (dP)^2 ,
\end{eqnarray*}
so that
\begin{equation}\label{EQU:(gdg)2n}
g_t^{-1} g'_t (g_t^{-1} d g_t )^{2n} = 2 \pi i (-4)^n \sin^{2n} ( \pi t) P (dP)^{2n}.
\end{equation}
Using the fact that
\[
\int_0^1 \sin^{2n} ( \pi t) dt = \frac{1}{\pi}\int_0^\pi \sin^{2n} t dt = \frac{1}{4^n}  \binom{2n}{n}
\]
we get 
\[
\int_0^1  g_t^{-1} g'_t\cdot (g_t^{-1} d g_t )^{2n} dt =  (-1)^n 2 \pi i \binom{2n}{n} P (dP)^{2n}
\]
and thus from Lemma \ref{oddCS}:
\begin{multline}\label{EQU:CS(gt)=Ch(nabla)}
CS(g_t) =\Tr \sum_{n\geq 0} \frac{(-1)^n}{(2 \pi i)^{n+1}} \frac{n!}{(2n)!} \int_0^1 (g_t^{-1} g_t') \cdot (g_t^{-1} dg_t)^{2n} dt \\
= \Tr\sum_{n\geq 0} \frac{1}{(2 \pi i)^{n}} \frac 1 {n!} P (dP)^{2n} = Ch(\nabla).\hspace{1.67in}
 \end{multline}
 
We now show the second statement in the theorem. Let $g : M \to \Omb U$  be given, that is $g_t : M \times S^1 \to U$. It suffices to show there is a smooth homotopy $g_t^s : M \times S^1\times I \to U$ such that $g_0^s=g_1^s=1$, $g_t^0 = g_t $, and $g_t^1(x) = \exp (2 \pi i P(x) t)$  where $P: M \to \mathcal{M}_{n \times n}$ is some smooth matrix valued function on $M$ and each $P(x)$ is a projection. Then, by Proposition \ref{prop:CSuptoExact}, $CS(g^1) - CS(g^0)$ is exact. The statement then follows from the previous calculation since $CS(g^1) = Ch(\nabla)$ where $\nabla$ is the connection on the bundle determined by $P$.

We will use ideas from proofs of Bott periodicity \cite{Mc, AP, B}. We recall that those proofs identify a geometric model for the space $BU \times \Z$ as follows. For a fixed Hilbert space $H$ one can consider the contractible space $E$ of  hermitian operators with eigenvalues in $[0,1]$. There is a map from $E$ to $U$ given by $A \mapsto \exp(2 \pi i A )$. 
After stabilization the pre-image of $1 \in \U$ is shown to be homotopy equivalent to $BU \times \Z$, see e.g. \cite{B}. Note that  $\exp(2 \pi i A ) =1$
if and only if the eigenvalues of $A$ are in $\{0,1\}$, i.e. $A$ is the projection onto a finite dimensional subspace.
Conversely, given such a subspace, there is a unique projection operator onto the given finite dimensional subspace.

Since $E$ is contractible via the path $tP$, it follows that the map $BU \times \Z \to \Omb U$,  defined on representatives $P$ by $\exp:P \mapsto \exp(2 \pi i P t)$, is a homotopy equivalence. 
By Whitehead's theorem, there is a map $\phi: \Omb U \to BU \times \Z$ which is a  homotopy inverse to this map.

The classifying space $BU \times \Z$ may be constructed as a limit of the finite dimensional smooth Grassmann manifolds $G_{k,n}$. In fact, there is a cell structure for $BU \times \Z$ so that any finite subcomplex is contained in some $G_{k,n}$.
This implies that any continuous map from a compact space into $BU \times \Z$ has image in $G_{k,n}$ for some $k$ and $n$.

It follows that the composition $\phi \circ g_t : M \to \Omb U \to BU \times \Z$ has image in the smooth manifold $G_{k,n}$ for some $k$ and $n$.  
By Whitney's approximation theorem, the continuous map $\phi \circ g_t$ is homotopic to a smooth map $P: M \to G_{k,n}$, where each $P(x)\in G_{k,n}$ is regarded as a projection onto a $k$-dimensional subspace. Then $g_t$ is homotopic to $\exp( 2 \pi i \phi(g_t))$, which is homotopic to $\exp(2 \pi i P t)$. Again, by compactness and Whitney's approximation theorem, we can arrange for the homotopy $g_t^s$ between $g_t^0 = g_t$ and 
$g_t^1= \exp(2 \pi i P t)$ to be smooth. This completes the proof.
\end{proof}

\begin{rmk} \label{rmk:CSCh}
An alternative interpretation of the last theorem is as follows. There is a canonical form on $BU \times \Z$ giving the universal Chern form, and a canonical form on $\Om U$, which is given by pullback of $\Theta$ on $\U$ along the evaluation map $\Omb U \times S^1 \to U$, and integrating along the fiber $S^1$. Furthermore, there is a map $BU \times \Z \to \Omb U$ given as an exponentiated loop, based on the specific Bott periodicity map in \cite{AP} and \cite{B}. The theorem says the pullback of the above form on $\Omb U$ along this map equals the universal Chern form on $BU \times \Z$, on the nose,
in the sense that this holds along any compact smooth plot $M \to BU \times Z$ (in the sense of Chen \cite{C}) and its induced plot $M \to \Omb U$. Consequently, the analogous pullback statement holds in the other direction, modulo exact, along any homotopy inverse of this map.
 \end{rmk}

The following important technical lemma will be used below to show that certain equivalence classes of elements of $Map(M, \U)$, to be defined later using the Chern Simons forms, have well defined abelian group structures.

\begin{lem} \label{lem:CSzero}
For any $g, h\in Map(M, \U )$, there exists a path $f_t \in Map(M \times I, \U )$ such that 
\[ 
f(0)= g \oplus h, \quad \quad f(\pi/2) = h \oplus g, \quad \textrm{and} \quad  CS(f_t) =0.
\]
Also, for any $g \in Map(M, \U )$, there exists a path $k_t \in Map(M \times I, \U  )$ such that 
\[ 
k(0)= g \oplus g^{-1}, \quad \quad k(\pi/2) = id, \quad \textrm{and} \quad  CS(k_t) =0.
\]
 \end{lem} 

\begin{proof}
We may assume that $g$ and $h$ are given by elements in $Map(M, U(n))$ for a common integer $n$, after possibly stabilizing one of the two maps. To prove both statements we will use the path $X:[0,\pi/2]\to U(2n)$, 
\[
X(t) = 
\begin{bmatrix}
\cos t & \sin t  \\
-\sin t   & \cos t 
\end{bmatrix},
\]
where each entry is an $n$-by-$n$ matrix, given by multiplying it with the $n$-dimensional identity matrix. Using this, we have
\[
J = X'(t) X(t)^{-1} = X(t)^{-1} X'(t) =
\begin{bmatrix}
0 & 1 \\
-1 & 0
\end{bmatrix}.
\]

For the first statement, consider the path $f_t\in Map(M\times [0,\pi/2], U(2n))$, 
\[
f_t  = 
X(t)FX(t)^{-1}
\]
where 
\[
F = \begin{bmatrix}
g & 0 \\
0 & h
\end{bmatrix}
\]
so that $f(0) = g \oplus h$ and $f(\pi/2) = h \oplus g$. Then, using the fact that $\frac{ \partial}{ \partial t} (X(t)^{-1}) = -X(t)^{-1}X'(t) X(t)^{-1}$, we get
\begin{eqnarray*}
f_t^{-1} f_t' &=& X(t)F^{-1}X(t)^{-1}( X'(t)FX(t)^{-1} - X(t)FX(t)^{-1}X'(t) X(t)^{-1}) \\
&=& X(t)F^{-1}JFX(t)^{-1} - J 
\end{eqnarray*}
and
\[
f_t^{-1} df_t = X(t)F^{-1}X(t)^{-1}( X(t)dFX(t)^{-1} ) =  X(t)F^{-1}dFX(t)^{-1}.
\]
So,
\begin{eqnarray*}
\Tr \left( f_t^{-1} f_t' (f^{-1} df_t)^{2n} \right) &=& \Tr \left( (X(t)F^{-1}JFX(t)^{-1} - J)(X(t)F^{-1}dFX(t)^{-1} )^{2n} \right) \\
& = & \Tr \left( X(t)F^{-1}JF (F^{-1}dF )^{2n} X(t)^{-1}\right. \\
&& \hspace{1.5in}\left.
- JX(t)(F^{-1}dF)^{2n}X(t)^{-1} \right)  \\
&=&  \Tr \left( J(dF F^{-1})^{2n} \right)  - \Tr \left( J(F^{-1}dF)^{2n} \right) = 0
\end{eqnarray*}
where in the second to last step we used that trace is invariant and $X^{-1}JX(t) = J$, and the final step we note the matrices are skew-symmetric. Using the explicit formula for $CS(f_t)$ from Lemma \ref{oddCS}, this proves the first statement.

For the second statement, consider the path $k_t\in Map(M\times [0,\pi/2], U(2n))$, 
\[
k_t  = G X(t) H X(t)^{-1},
\]
where 
\[
G = \begin{bmatrix}
g & 0 \\
0 & 1
\end{bmatrix}
\quad \quad
H = \begin{bmatrix}
1 & 0 \\
0 & g^{-1}
\end{bmatrix}.
\]
Then $k(0) = g \oplus g^{-1}$ and $k(\pi/2) = I_{2n}$. We need to show that $CS(k_t)=0$.

Using again that $\frac{ \partial}{ \partial t} (X(t)^{-1}) = -X(t)^{-1}X'(t) X(t)^{-1}$,  we calculate
\begin{eqnarray*}
k_t^{-1} k_t' &=& X(t)H^{-1}X(t)^{-1} G^{-1} \left( G X'(t)H X(t)^{-1}\right. \\
&& \hspace{1.5in}\left. - G X(t) H X(t)^{-1}X'(t) X(t)^{-1}\right) \\
&=&  X(t)H^{-1}JHX(t)^{-1} - J 
\end{eqnarray*}
and
\begin{eqnarray*}
k_t^{-1} dk_t &=& X(t)H^{-1}X(t)^{-1}G^{-1} \left(  dGX(t)H X(t)^{-1}  + G X(t) dH X(t)^{-1} \right) \\
&=&   X(t)H^{-1}X(t)^{-1}G^{-1} dGX(t)H X(t)^{-1}  + X(t)H^{-1} dH X(t)^{-1} 
\end{eqnarray*}
which implies
\begin{eqnarray*}
(k_t^{-1} dk_t)^{2n} &=&   X(t)H^{-1}\left( X(t)^{-1}G^{-1} dGX(t) + dH H^{-1} \right)^{2n} HX(t)^{-1} 
\end{eqnarray*}
We will evaluate this expression by showing that 
\[
\left( X(t)^{-1}G^{-1} dGX(t) + dH H^{-1} \right)^{2n} 
=
(g^{-1}dg)^{2n}
\begin{bmatrix}
\cos^{2n}(t) & 0 \\
0 & \cos^{2n}(t)
\end{bmatrix}.
\]
Indeed, one calculates that
\[
X(t)^{-1}G^{-1} dGX(t) 
=
g^{-1}dg 
\begin{bmatrix}
\cos^2(t) & \cos(t) \sin (t) \\
\cos(t)\sin (t)   & \sin^2(t) 
\end{bmatrix}
\]
and, since $d(g^{-1})=-g^{-1}dgg^{-1}$, 
\[
dH H^{-1} 
=
g^{-1}dg 
\begin{bmatrix}
0 & 0 \\
0  &-1
\end{bmatrix},
\]
which gives 
\[
X(t)^{-1}G^{-1} dGX(t) + dH H^{-1} 
=
(g^{-1}dg)
\begin{bmatrix}
\cos^2(t) & \cos(t) \sin (t) \\
\cos(t) \sin (t)  & -\cos^2(t) 
\end{bmatrix}.
\]
Since
\[
\begin{bmatrix}
\cos^2(t) & \cos(t) \sin (t) \\
\cos(t) \sin (t)  & -\cos^2(t) 
\end{bmatrix}
^2 
=
\begin{bmatrix}
\cos^2(t) & 0 \\
0 & \cos^2(t) 
\end{bmatrix}
\]
it follows that 
\[
\left( X(t)^{-1}G^{-1} dGX(t) + dH H^{-1} \right)^{2n} 
=
(g^{-1}dg)^{2n}
\begin{bmatrix}
\cos^{2n}(t) & 0 \\
0 & \cos^{2n}(t)
\end{bmatrix}.
\]

Finally we have
\begin{multline*}
\Tr \left( k_t^{-1} k_t' (k^{-1} dk_t)^{2n} \right) 
= \Tr \Bigg( \left( X(t)H^{-1}JHX(t)^{-1} - J \right)  
\\
 \hspace{1.2in}
\cdot \left( X(t)H^{-1} 
(g^{-1}dg)^{2n}
\begin{bmatrix}
\cos^{2n}(t) & 0 \\
0 & \cos^{2n}(t)
\end{bmatrix} HX(t)^{-1} \right) \Bigg) 
\\
=
\Tr \left( J (g^{-1}dg)^{2n} \begin{bmatrix}
\cos^{2n}(t) & 0 \\
0 & \cos^{2n}(t)
\end{bmatrix}  
\right)  \hspace{.4in}
\\
-
\Tr \left(
HJ H^{-1}  (g^{-1}dg)^{2n} \begin{bmatrix}
\cos^{2n}(t)  & 0 \\
0 & \cos^{2n}(t)
\end{bmatrix}  
 \right),
\end{multline*}
where we have used the fact that trace is invariant and $X^{-1}(t) J X(t) = J$. Since
\[
HJ H^{-1} 
=
\begin{bmatrix}
0 & g \\
-g^{-1} & 0
\end{bmatrix}, 
\]
we obtain that
\begin{multline*}
\Tr \left( k_t^{-1} k_t' (k^{-1} dk_t)^{2n} \right) 
=
\Tr \left( (g^{-1}dg)^{2n} \begin{bmatrix}
0 & \cos^{2n}(t) \\
-\cos^{2n}(t) & 0
\end{bmatrix}  
\right) \\
-
\Tr \left(
  (g^{-1}dg)^{2n} \begin{bmatrix}
 0 & g \cos^{2n}(t)  \\
-g^{-1} \cos^{2n}(t)  & 0
\end{bmatrix}  
 \right)=0
\end{multline*}
By Lemma \ref{oddCS} this shows that $CS(k_t)=0$, which completes the proof of the lemma.
 \end{proof}

\section{Differential extensions}

In \cite[Definition 2.1]{BS2}, Bunke and Schick give a definition for the ``differential extension'' (formerly known as a ``smooth extension'') of any generalized cohomology theory. For the purposes of this paper, we will restrict the discussion to case of complex K-theory, for which case the data becomes $\Z_2$-graded (see remarks following Definition 2.1 in \cite{BS2}). Let 
$[Ch]: K^0(M) \to H^{\textrm{even}}(M;\R)$ denote the ordinary even Chern character, which is induced by the
map $[Ch]: K^0(M) \to \Om^{\textrm{even}}(M;\R) / Im(d)$ defined in Section \ref{SEC:odd-chern-form}. 

\begin{defn} \label{defn:diffext}
A \emph{differential extension of K-theory} is a contravariant functor $\hat K$ from the category of compact smooth manifolds
(possibly with boundary) to the category of $\Z_2$-graded abelian groups, together with 
natural transformations
\begin{enumerate}
\item $R : \hat K^* (M) \to \Om^*_{cl} (M;\R)$
\item $I :  \hat K^* (M)  \to K^*(M)$
\item $a: \Om^{*-1}(M;\R) / Im(d) \to \hat K^{*}(M)$
\end{enumerate}
such that
\begin{enumerate}
\item The following diagram commutes
\[
\xymatrix{
 & K^{*}(M) \ar[rd]^{[Ch]} & \\
 \hat K^{*}(M) \ar[ru]^I \ar [rd]^{R} & &H^{*}(M) \\
 & \Om^{*}_{cl} (M) \ar[ru]_-{\textrm{deRham}} &
}
\]
\item $R \circ a = d$
\item The following sequence is exact
\[
\xymatrix{
K^{*-1}(M) \ar[r]^-{[Ch]} &  \Om^{*-1}(M;\R) / Im(d) \ar[r]^-{a} & \hat K^{*}(M) \ar[r]^{I} & K^*(M) \ar[r]^{\quad 0} & 0
} 
\]
\end{enumerate}
\end{defn}

Several constructions have been given producing differential extensions of K-theory, e.g. \cite{HS}, \cite{BS}, \cite{FL}.
Bunke and Schick clarify in \cite{BS3} that the axioms above do not uniquely determine the differential extension and illuminate the fundamental role played by an $S^1$-integration map (which may be regraded as the smooth analogue of a suspension isomorphism). Additionally, one can ask for a compatible ring structure on a differential extension, referred in \cite{BS3} as a \emph{multiplicative structure}, and in fact this structure implies the existence of an $S^1$-integration map. Moreover, such additional structure determines a differential extensions of K-theory uniquely by a unique natural isomorphism \cite{BS3}.

\begin{rmk}
Note that the data of a differential extension of K-theory splits into an even and odd part, according to the domain of the functors $\hat K^*$, $R$ and $I$, and the range of $a$.  It therefore makes sense to refer to the even or odd part of a differential extension, or equivalently to a differential extension of the even or odd part of K-theory, and we will do so in what follows.
\end{rmk}

An elementary construction of the even part of a differential
extension of K-theory was given by \cite{SS} by Simons and Sullivan, which is the Grothedieck group of isomorphism classes
of vector bundles with connection, up to the equivalence relation of Chern-Simons exactness.
While the language of differential extension was not used in \cite{SS}, almost all of the data and conditions in Definition \ref{defn:diffext} are apparent there, namely $R$ is the even Chern character given by Equation \eqref{eq:Ch(nabla)} and $I$ is the forgetful map.
The only map that is perhaps not explicit is the map $a: \Om^{\textrm{odd}}(M;\R) / Im(d) \to \hat K^{0}(M)$. It may be defined as the composition
\[
\xymatrix{
\Om^{\textrm{odd}}(M;\R) / Im(d) \ar[r]^-{\pi} & \left( \Om^{\textrm{odd}}(M;\R) / Im(d) \right) / Im(Ch) \ar[r] & \hat K^0(M)
}
\]
where $Ch: K^{-1}(M) \to \Om^{\textrm{odd}}(M;\R) / Im(d)$ is the odd Chern character, $\pi$ is the projection map, and the last map is shown in \cite{SS} to be an inclusion. It follows immediately that  sequence in Definition \ref{defn:diffext}  is exact.

It follows from a theorem of Bunke and Schick \cite{BS2} that this even part of a differential extension of K-theory, is naturally isomorphic  to the even part of any other differential extension. In particular, this is a construction of the ring of even differential K-theory (which is unique via a unique natural isomorphism).

We emphasize some particularly nice properties of this extension. Firstly, the geometric data used to construct this extension are rather small.  In particular, it may be relatively easier to build maps out of this differential extension. Secondly, this differential extension admits a natural refinement, defined using a lifting of the Chern Simons form to the free loopspace of the manifold,  \cite{TWZ}. In the next section we will construct the odd part of a differential extension of K-theory, from the 
geometric data given by the smooth mapping space $Map(M,U)$.

\section{The odd part of a differential extension} \label{sec:elemmodel}

Recall that for a smooth manifold $M$, $K^{-1}(M)$ may be defined  as  the set homotopy classes of maps from $M$ to $\U$.
This defines a contravariant functor to abelian groups under the operation of block sum.
    We now introduce an equivalence relation on the set $Map (M, \U )$ which is finer than homotopy equivalence, and which will be shown to define an abelian group  $\hat K^{-1}(M)$, giving the odd part of a differential extension of $K^{-1}$. 
 \begin{defn}
For $g_0, g_1 \in Map (M, \U )$ we say $g_0 \sim g_1$ if there is $g_t: M \times I \to U$ such that $CS(g_t)$ is $d$-exact.  This defines an equivalence relation, and we denote the set of equivalence classes by $\hat K^{-1} (M)$.

For a morphism $f:M'\to M$ of smooth manifolds, we define $f: \hat K^{-1} (M)\to \hat K^{-1} (M')$ by $f([g])=[g\circ f]$, which is well defined since $CS(g_t\circ f)=f^*(CS(g_t))$.
\end{defn}  

Note in particular that if $g_0 \sim g_1$ then there is a $g_t$ so that the degree zero part of $CS(g_t)$ is $0$, so that $\Tr  \int_0^1 (g_t^{-1} g_t') dt = 0$.

\begin{prop} \label{prop:group}
The block sum $\oplus$ induces a well defined abelian group structure on $\hat K^{-1} (M)$.
\end{prop}

\begin{proof} 
If $g_0 \sim g_1$ and $h_0 \sim h_1$ then for some paths $g_t$ and $h_t$ we have $CS(g_t)$ and $CS(h_t)$ are exact.
Then $g_0 \oplus h_0  \sim g_1 \oplus h_1$, since by Proposition \ref{prop:CSprop}, the path $g_t \oplus h_t$ satsfies
\[
CS(g_t \oplus h_t) = CS(g_t) + CS(h_t),
\]
which is exact. 

The constant map to  $1 \in \U$ is an additive identity for the sum. Also, Lemma \ref{lem:CSzero} shows that the equivalence class of $g$ has an inverse, given by the equivalence class of $g^{-1}$.  Finally, the product is abelian again by Lemma \ref{lem:CSzero}  since there is a path $f_t$ from $g \oplus h$ to  $h \oplus g$ such that $CS(f_t) = 0$.
 \end{proof}

\begin{prop}
There is a commutative diagram of group homomorphisms
\[
\xymatrix{
 & K^{-1}(M) \ar[rd]^{[Ch]} & \\
 \hat K^{-1}(M) \ar[ru]^I \ar [rd]^{R=Ch} & &H^{\textrm{odd}}(M) \\
 & \Om^{\textrm{odd}}_{cl} (M) \ar[ru]_-{\textrm{deRham}} &
}
\]
where $I([g])$ equals the homotopy class of any representative $g$ for $[g]$. The maps $R$ and $I$ are natural transformations of functors.
\end{prop}

\begin{proof}
The map $Ch: \hat K^{-1}(M) \to \Om^{\textrm{odd}}_{cl} (M) $ given by $Ch([g]) = Ch(g)$ is well defined since
for $g_0 \sim g_1$ we have $CS(g_t)$ is exact for some $g_t$ so that  $Ch(g_1) - Ch(g_0) = dCS(g_t) = 0$.
This is a group homomorphism, since $Ch(g \oplus h) = Ch(g) + Ch(h)$, by Corollary \ref{cor:Chadd}. 

The map $I: \hat K^{-1}(M)  \to K^{-1}(M)$ is the forgetful map which sends $CS$-equivalence classes to the equivalence class determined by path components. This map is a group homomorphism since addition on $K^{-1}(M)$ may be defined by using the 
block sum operation, as there is a path from $g \oplus h$ to $gh$ given by 
\[
f(t) = 
\begin{bmatrix}
g & 0 \\
0 & 1
\end{bmatrix}
X(t)
\begin{bmatrix}
1 & 0 \\
0 & h
\end{bmatrix}
X(t)^{-1}
\]
where $X(t) = \begin{bmatrix}
\cos(t)& \sin(t) \\
-\sin(t) & \cos(t)
\end{bmatrix} 
$ is a map $[0,\pi/2]\to U(2n)$ as in Lemma \ref{lem:CSzero}. It is straightforward to check that the diagram commutes.
\end{proof}
 
The remaining data of a differential extension consists of a natural transformation $a$, i.e. for each $M$ a map
\[
a : \Om^{\textrm{even} } (M) / Im(d) \to \hat K^{-1}(M) 
\]
such that $Ch\circ a=d$, and so that we obtain an exact sequence 
\[
\xymatrix{
 K^0(M) \ar[r]^-{[Ch]} & \Om^{\textrm{even} } (M) / Im(d) \ar[r]^-{a} &  \hat K^{-1}(M) \ar[r]^I & K^{-1}(M) \ar[r] & 0
}.
\]
Clearly $I$ is surjective. In order to define $a$, it is sufficient to define an isomorphism
 \[
 \beta : Ker(I) \to \left( \Om^{\textrm{even} } (M) / Im(d) \right) / Im([Ch])
  \]
 for then we may let $a$ be the composition of the projection $\pi$ with $\beta^{-1}$,
 \[
\xymatrix{
a: \Om^{\textrm{even} } (M) /  Im(d) \ar[r]^-{\pi} & \left( \Om^{\textrm{even} } (M) /  Im(d) \right) / Im([Ch]) \ar[r]^-{\beta^{-1}} & Ker(I) \subset \hat{K}^{-1}(M)
}
\]
and we have $Ker(a) = Im([Ch])$ and $Im(a) = Ker(I)$.

 There is a natural candidate  for this map $\beta$, as follows. 
Suppose $[g] \in  Ker(I) \subset \hat K^{-1}(M)$. 
Then, for any $g \in [g]$, there is a (non-unique) $g_t: M \times I \to U$ such that $g_1 = g$ and $g_0$ is the constant map 
$M \to \U$ to the identity of $\U$, \emph{i.e.}
\[
Ker(I) = \{ [g] | \, \textrm{for some $g \in [g]$ there is a path $g_t$ such that $g_1 = g$ and $g_0=1$}\}.
\]
Define
\[
\beta([g]) = CS(g_t)   \quad  \in \left( \Om^{\textrm{even} } (M) / Im(d) \right) / Im([Ch]) 
\]
where $g_t : M \times I \to U$ is a choice of map satisfying $g_0 = 1$ (the constant map to the identity in $U$), and 
$g_1 \in [g]$.

We first show this map $\beta$ is well defined. If $g, h \in [g]$ are two such choices, then there is a path $k_t$ such that $k_0 = g$, $k_1= h$ and $CS(k_t)$ is exact. Then for any path $g_t$ from $1$ to $g$, and path $h_t$ from $1$ to $h$, we can consider the loop based at the identity $1$, defined by taking the composition of paths $g_t * k_t * h_t^{-1}$. By Proposition \ref{prop:CSprop} and Theorem \ref{thm:CSCh} we have that, modulo exactness, $CS(g_t * k_t * h_t^{-1}) = CS(g_t) + CS(k_t) - CS(h_t) \in Im([Ch])$. Since $CS(k_t)$ is exact, this shows $CS(g_t) =  CS(h_t) \in \left( \Om^{\textrm{even} } (M) / Im(d) \right) / Im([Ch])$. This shows that $\beta$ does not depend in the choice of $g \in [g]$. 

Now, for fixed choice $g \in [g]$ suppose $g_t$ and $h_t$ are two paths starting at $1$ and ending at $g$. Then considering $g_t * h_t^{-1}: M \times S^1 \to U$, by Proposition \ref{prop:CSprop} and Theorem \ref{thm:CSCh} we have that, modulo exactness, 
\[
CS(g_t * h_t^{-1}) = CS(g_t) - CS(h_t) \in Im([Ch])
\]
so that $CS(g_t) = CS(h_t) \in \left( \Om^{\textrm{even} } (M) / Im(d) \right) / Im(Ch) $.

It is clear that $\beta$ is a homomorphism, since $\beta ([g \oplus h]) = CS(g_t \oplus h_t) = CS(g_t) + CS(h_t) = \beta([g]) + \beta([h])$, and we now show $\beta$ is injective:
If $\beta([g]) = 0$ then 
\[
CS(g_t) = Ch(\nabla) + dZ
\] 
for some some path $g_t$ from $1$ to $g$, some connection $\nabla$, and some form $Z$.  By Theorem \ref{thm:CSCh} there is loop $h_t$ such that $CS(h_t) = Ch(\nabla)$. Then 
 $k_t= h_t^{-1} * g_t : M \times I \to U$ satisfies $k_0 = 1$, $k_1 = g$, and $CS(k_t) = CS(g_t) -Ch(\nabla) = dZ$ is exact, so that $[g] = 0$.

 It therefore remains to show that $\beta$ is surjective. We reduce the statement as follows.
 \begin{lem} \label{lem:recast}
The map 
\[
\beta: Ker(I) \to  \left( \Om^{\textrm{even} } (M;\R) / Im(d) \right) / Im([Ch])  
\]
is surjective if and only if the following two statements both hold:
\begin{enumerate}
\item Every exact odd form $dZ$ on $M$ is given by $dZ=Ch(g)$ for some $g \in Map^o(M, U)$ in the connected component of the identity $1$.
\item For all $X \in \left( \Om_{cl}^{\textrm{even} } (M;\R) / Im(d) \right) / Im([Ch])$  there is some 
$g_t : M \times I \to U$ satisfying $g_0 = 1$ and
\[
X =  CS(g_t) \in \left( \Om^{\textrm{even} } (M;\R) / Im(d) \right) / Im([Ch])
\]

\end{enumerate}
\end{lem} 
 \begin{proof}
 The second statement clearly follows from the surjectivity of $\beta$. For the first statement, if $Y$ is an exact odd form, choose $Z$ so that $Y = dZ$ and let $[Z] \in \left( \Om^{\textrm{even} } (M;\R) / Im(d) \right) / Im([Ch])$  be the image of $Z$ under the quotient map. Then by assumption we can write 
 \[
 [Z] = CS(g_t) \in \left( \Om^{\textrm{even} } (M) / Im(d) \right) / Im([Ch])
 \]
 for some $g_t : M \times I \to U$ such that $g_0 = 1$. Since exterior derivative is well defined on 
 $\left( \Om^{\textrm{even} } (M;\R) / Im(d) \right) / Im([Ch])$ and independent of choice of representatives, we have
 \[
 Y = dZ= d CS(g_t) = Ch(g_1)-Ch(g_0)=Ch(g_1),
 \]
 and $g_1$ is in the connected component of the identity.
 
 Conversely, suppose $X \in  \left( \Om^{\textrm{even} } (M) / Im(d) \right) / Im([Ch])$. Then $d X \in \Om^{\textrm{odd} }(M;\R)$ is well defined and, by the first assumption, we may write $dX = Ch(g) = dCS(g_t)$ where $g_t$ is any choice of path from the identity to the given $g \in Map^o(M, U)$.  So, $X - CS(g_t)$ is a closed even form. Let 
 $[X - CS(g_t)] \in  \left( \Om_{cl}^{\textrm{even} } (M) / Im(d) \right) / Im([Ch])$ be the image of $X - CS(g_t)$ under the quotient map.
 Then, by the second assumption, 
 \[
 [X - CS(g_t)] = CS(h_t) \in  \left( \Om^{\textrm{even} } (M) / Im(d) \right) / Im([Ch])
 \]
  for some $h_t$ satisfying $h_0=1$.

 Therefore, $X = CS(g_t \oplus h_t) \in \left( \Om^{\textrm{even} } (M) / Im(d) \right) / Im([Ch])$, so that
 $\beta( [g_1 \oplus h_1]) = X$,  and 
 $g_t \oplus h_t$ satisfies $g_0 \oplus h_0= 1$, so $[g_1 \oplus h_1] \in Ker(I)$.
 \end{proof}
  
The first condition in Lemma \ref{lem:recast} follows from Corollary~\ref{cor:ChSurjOnExact}. 
We now show that condition (2) holds.

\begin{lem} \label{cor:CSontoClosedModExact}  
Every closed real valued even differential form on $M$ equals, modulo exact forms, $CS(g_t)$ for some path $g_t$ satisfying $g_0 = 1$.
\end{lem}
\begin{proof} 
Since the Chern character map  $Ch: K^0(M) \to H^{\textrm{even}}(M;\R)$ becomes as isomorphism after tensoring with $\R$, we can find some bundles with connection $(E_1,\nabla_1),\dots,(E_r,\nabla_r)$, such that the forms $\{Ch(\nabla_j)\}_{j=1\dots r}$ give a basis for all of $H^{\textrm{even}}(M;\R)$. 
Given these $\nabla_1,\dots, \nabla_r$, we can use the path $g^j_t=e^{2 \pi i P_j t}$ for each $j$, where $P_j$ is a suitable choice of a projection operator associated to $\nabla_{j}$ as in the proof of Theorem \ref{thm:CSCh}. Restricting any $g^j_t$ to an interval $[0,s]\subset [0,1]$, we see from Lemma \ref{oddCS} and Equations \eqref{EQU:(gdg)2n} and \eqref{EQU:CS(gt)=Ch(nabla)} 
in the proof of Theorem \ref{thm:CSCh}, that
 \begin{eqnarray*}
CS\left( g^j_t \Big|_{[0,s]}  \right)&=& 
 \Tr \sum_{n\geq 0} \frac{(-1)^n}{(2 \pi i)^{n+1}}  \frac{n!}{(2n)!} \int_0^s ((g^j_t)^{-1} (g^j_t)') \cdot ((g^j_t)^{-1} d(g^j_t))^{2n} dt \\
 &=& \Tr \sum_{n\geq 0} \frac{(-1)^n}{(2 \pi i)^{n+1}} \frac{n!}{(2n)!} \left( 2 \pi i (-4)^n \int_0^s  \sin^{2n} ( \pi t)  dt \right) P_j (dP_j)^{2n} \\
 &=&   \sum_{n\geq 0}\frac{ \int_0^s  \sin^{2n} ( \pi t)  dt }{\int_0^1  \sin^{2n} ( \pi t)  dt} 
\cdot  Ch(\nabla_{j})^{(2n)}
\end{eqnarray*}
where $Ch(\nabla_{j})^{(2n)}$ denotes the degree $2n$ part of $Ch(\nabla_{j})$. Using that $\int_0^1  \sin^{2n} ( \pi t)  dt={\frac{1}{4^{n}}\binom{2n}{n}}$, we denote by $f_{2n}(s)=\frac{4^n}{\binom{2n}{n}} \cdot { \int_0^s  \sin^{2n} ( \pi t)  dt }$, so that
\[
CS\left( g^j_t \Big|_{[0,s]}  \right)=\sum_n f_{2n}(s)\cdot  Ch(\nabla_{j})^{(2n)}.
\]

We note that the various degrees scale differently. Nevertheless, using this fact above, it is fairly striaghtforward to show using only algebra that any rational linear combination of Chern forms can be written, modulo exact, as $CS(g_t)$ for some path $g_t$. 
But, to obtain all real linear combinations of Chern forms, it seems we must give an analytic argument, which we do now.

Each component $Ch(\nabla_{j})^{(2n)}$ is a closed form, so that it can be written in terms of the generating set
\begin{equation}\label{EQ:bnlj}
Ch(\nabla_{j})^{(2n)}=\sum_{\ell=1}^r b_j^{n,\ell} Ch(\nabla_\ell)\quad\quad
\text{modulo exact forms}
\end{equation}
for some real numbers $b_j^{n,\ell}$. We may take the block sum of all the $g^j_t$ restricted to $[0,s_j]$, by extending each to be constant on $[s_j,1]$, respectively, and then the Chern-Simons form of this is given by 
\begin{eqnarray*}
\sum_{j=1}^r CS\left( g^j_t \Big|_{[0,s_j]}  \right)&=&\sum_{j=1}^r\sum_{n}\sum_{\ell=1}^r f_{2n}(s_j) b_j^{n,\ell} Ch(\nabla_\ell) \\
&=& \sum_{\ell=1}^r\left(\sum_{j=1}^r F^\ell_{j}(s_j) \right) Ch(\nabla_\ell), 
\end{eqnarray*}
where we have set $F^\ell_j(s)=\sum_n b_j^{n,\ell}f_{2n}(s)$. In order to show that we can obtain any real linear combination of Chern forms from above expression, e.g. $\sum_j c_j\cdot Ch(\nabla_j)$ for some constants $c_j$, we need to solve the following system of $r$ equations in the $r$ variables $s_1,\dots, s_r$:
\begin{equation*}
\left\{
\begin{matrix}
c_1&=& \sum\limits_{j=0}^r F^{1}_{j}(s_j) \\
\vdots&& \vdots \\
c_r&=& \sum\limits_{j=0}^r F^{r}_{j}(s_j)
\end{matrix}
\right.
\end{equation*}
Let $G:\R^r\to \R^r$ be given by $(s_1,\dots, s_r)\mapsto(\sum_{j=0}^r F^{\ell}_{j}(s_j))_\ell$. By the inverse
function theorem, it suffices to prove that the matrix $M= dG = \big(\frac{\partial}{\partial s_j}(\sum_{j=0}^r F^{\ell}_{j}(s_j)) \big)_{\ell,j}$ has a non-zero determinant for some $s_1,\dots, s_r$. Then there is a neighborhood on which the function $G$ is invertible, so we may obtain any closed form $\sum_j c_j\cdot Ch(\nabla_j)$ for $(c_1,\dots,c_r)$ in a small
cube contained in $I_1\times\dots\times I_r$. Then, using block sums and inverses, with $CS(g_t \oplus h_t) = CS(g_t) + CS(h_t)$ and $CS(g^{-1}_t) = - CS(g_t)$, we may obtain any linear combination $\sum_j c_j\cdot Ch(\nabla_j)$ for $(c_1,\dots,c_r)\in \R^r$ as some $CS(g_t)$ (modulo exact forms).

Note that
\[
\frac{\partial}{\partial s_j}\big(\sum_{j=0}^r F^{\ell}_{j}(s_j)\big)=\frac d {ds_j} F_j^\ell(s_j) =\sum_n b_{j}^{n,\ell}\cdot \frac{4^n}{\binom{2n}{n}}\cdot \sin^{2n}(\pi s_j)
\]
 so that
\[ 
M
=
\left(\sum_n b_{j}^{n,\ell}\cdot \frac{4^n}{\binom{2n}{n}}\cdot \sin^{2n}(\pi s_j)\right)_{\ell,j} 
\]

Let $B_n = (b^{n,\ell}_j)_{\ell,j}$ for $0 \leq n \leq d$.  From the definition of $b^{n,\ell}_j$, we see by summing \eqref{EQ:bnlj} over all $n$ that $Ch(\nabla_j)=\sum_{\ell=1}^r\sum_n b^{n,\ell}_j Ch(\nabla_\ell)$, so that 
\[
B_0 + \cdots + B_d = Id
\] 
Taking determinant and using multi-linearlity we have
\[
\sum_{{0 \leq i_1, \ldots , i_r \leq n}}
\,\, \sum_{ \sigma \in S_r} \sgn(\sigma)\cdot  b^{i_1,\sigma(1)}_1 \cdots b^{i_r,\sigma(r)}_r 
=1
\]
This shows that not all of the summands $\left(\sum_{ \sigma \in S_r} \sgn(\sigma)\cdot  b^{i_1,\sigma(1)}_1 \cdots b^{i_r,\sigma(r)}_r \right)$ are zero.
On the other hand, 
\[
M= C_0+ \dots +C_d
\]
where 
\[ 
C_n = \frac{4^n}{\binom{2n}{n}} \cdot B_n \cdot
\left(
\begin{array}{ccc}
\sin^{2n}(\pi s_1) &  & \\
 & \ddots & \\
 & &  \sin^{2n}(\pi s_r)
\end{array}
\right)
\]
so that
\[
det(M) = \sum_{{1 \leq i_1, \ldots , i_r \leq n}}
\,\, \sum_{ \sigma \in S_r} \sgn(\sigma)\cdot  b^{i_1,\sigma(1)}_1 \cdots b^{i_r,\sigma(r)}_r 
\left( \prod_{k=1}^r \frac{4^{i_k}}{\binom{2i_k}{i_k}}\cdot \sin^{2i_k}(\pi s_k)\right).
\]
The collection of products of sine functions appearing as the summands above are linearly independent over $\R$. Since at least one coefficient $\left( \sum_{ \sigma \in S_r} \sgn(\sigma)\cdot  b^{i_1,\sigma(1)}_1 \cdots b^{i_r,\sigma(r)}_r  \right) $ is non-zero, this shows the determinant is non-zero for
some $s_1 , \ldots , s_r$. This completes the proof.
\end{proof}

\begin{cor} \label{cor:betasurj}
The map 
\[
\beta: Ker(I) \to \Om^{\textrm{even} } (M) / \left( Im([Ch]) + Im(d) \right)
\]
is surjective. 
  \end{cor}
  \begin{proof}
The two conditions in Lemma \ref{lem:recast} follow from Corollary~\ref{cor:ChSurjOnExact} and Lemma~\ref{cor:CSontoClosedModExact}, respectively.
 \end{proof}

Finally, having shown $\beta$ is invertible, we may define $a = \beta^{-1} \circ \pi$ as explained above, giving us the desired exact sequence
 \[
\xymatrix{
 K^0(M) \ar[r]^-{[Ch]} & \Om^{\textrm{even} } (M) / Im(d) \ar[r]^-{a} &  \hat K^{-1}(M) \ar[r]^I & K^{-1}(M) \ar[r] & 0.
}
\]

\begin{thm} \label{thm:diffext}
The functor $M \mapsto K^{-1}(M)$, together with the maps $R=Ch, I$, and $a$ above, define a differential extension of odd K-theory. 
\end{thm}
\begin{proof}
It only remains to show that $R \circ a = d$. For a given form $X\in \Om^{\textrm{even} } (M) / Im(d)$, we have $\pi(X) = \beta([g_1])= CS(g_t)$ 
for some  path $g_t$ starting at $g_0=1$. Then $Ch ( a ( X))  =Ch([g_1])= Ch(g_1) =  Ch(g_1) - Ch(g_0) = dCS(g_t) =dX$, as $d\pi(X)=dX$ is well-defined since $d(Im([Ch]))=0$.
\end{proof}

We close this section with a slightly stronger result than was needed here. It will be used in the appendix.
  
 \begin{cor} \label{cor:CSsurj}
 Let $Map^o(M \times I, U) = \{ g_t: M \times I \to U | \, g_0 = 1 \}$.
The map 
\[
CS: Map^o(M \times I, U) \to \Om^{\textrm{even} } (M) / \left(Im(d) \right)
\]
is surjective. 
  \end{cor}

 \begin{proof}
 Given $[X] \in \Om^{\textrm{even} } (M) / \left(Im(d) \right)$ we know by Corollary \ref{cor:betasurj} that  $X = \beta([g]) = CS(g_t)$ modulo $Im(Ch) + Im(d)$ for some $g_t \in Map^o(M \times I, U)$. In other words,  we have $X = CS(g_t) + Ch(\nabla)$ modulo exact forms. But $Ch(\nabla)$ is closed, and so by Theorem \ref{thm:CSCh} we may write $Ch(\nabla) = CS(h_t)$ for some $h_t \in Map^o(M \times S^1, U) \subset Map^o(M \times I, U)$. Therefore, $X = CS(g_t) + CS(h_t) = CS(g_t \oplus h_t)$ modulo exact forms is in the image of $CS|_{Map^o(M \times I, U)}$, as claimed.
 \end{proof}

\section{Calculation for a point}
 
One can calculate directly from the exact sequence property (3) in definition \ref{defn:diffext} that the odd differential extension of a point is $\hat K^{-1}(pt) \cong S^1 = \R / \Z$. In this short section we give an alternative straightward calculation of this result, that illuminates the meaning of $CS$-equivalence in this context, which could also be useful in other examples.
 
 \begin{lem} \label{lem:pt}
 Conjugate elements of $\U$ are $CS$-equivalent in $Map(pt,\U)$. In particular, after diagonalization, every element of $\hat K^{-1}(pt)$ is represented by a diagonal matrix. 
 
 Moreover, if $D$ is diagonal then $D$ is $CS$-equivalent to the one by one matrix given by $[det(D)]$.
 \end{lem}

\begin{proof} Let $A \in U$. For any $P \in U$ we may choose a path $P_t$ such that $P_0 = 1$ and $P_1 = P$.
Then $g_t = P_t A P_t^{-1}$ is a path from $A$ to $PAP^{-1}$.  We calculate $CS(g_t)$, which in this case is a function (since we are over a point $M=pt$),
\begin{eqnarray*}
CS(g_t) &=& \Tr\left(\int_0^1 g_t^{-1} g_t' dt \right)  \\ 
&=& 
\Tr\left(\int_0^1 P_t A^{-1} P_t^{-1} \left(  P_t' A P_t^{-1} - P_t A P_t^{-1} P_t' P_t^{-1} \right) dt \right) \\
&=&\Tr\left(\int_0^1 P_t^{-1} P_t' - P_t' P_t^{-1} dt \right) \\
&=& 0.
\end{eqnarray*}

For the second statement, by induction it suffices to show that if $D$ is two by two and diagonal, with entries $g$ and $h$,
then $D \sim [gh]$, where $[gh]$ is a one by one matrix. Consider
\[
g_t = 
G
X(t)
H
X(t)^{-1}
\]
where 
\[
X(t) = \begin{bmatrix}
\cos(t) & \sin(t) \\
-\sin(t) & \cos(t)
\end{bmatrix}
\quad \quad
G =
\begin{bmatrix}
g & 0 \\
0 & 1
\end{bmatrix}
\quad \quad
H =
\begin{bmatrix}
1 & 0 \\
0 & h
\end{bmatrix}
\]
so that $g_0 = D$ and $g_{\pi / 2} = [gh]$. For this path we have $CS(g_t) = 0$ since 
\begin{multline*}
 \Tr \bigg( XH^{-1}X^{-1} G^{-1} \Big(  GX' HX^{-1}  - GX HX^{-1} X'  X^{-1} \Big)  \bigg) 
 \\ =  \Tr  \left( X^{-1} X'  -  X' X^{-1} \right) = 0.
\end{multline*}
This completes the proof of the lemma.
\end{proof}

 \begin{prop}
 The function $det : \hat K^{-1}(pt) \to S^1$ is a well defined group isomorphism.
 \end{prop}
 
\begin{proof} If $A,B \in U$ and $A \sim B$ then by the above Lemma \ref{lem:pt}, $[det(A)] \sim [det(B)]$. So the function is well defined if  
we show $[e^{i \theta_0}] \sim [e^{i \theta_1}]$ implies $e^{i \theta_0} = e^{i \theta_1}$. If 
$e^{i \theta_0} \neq e^{i \theta_1}$ then for any path $g_t \in U$ such that $g_0 = e^{i \theta_0}$ and 
$g_1 = e^{i \theta_1}$ we have 
\[
CS(g_t) = \Tr \left( \int_0^1 g_t^{-1} g_t'  \right) dt \neq 0.
\]
The function is clearly surjective, and it is injective since if $det(A) = 1$ then $A \sim [det(A)] = [1]$ so that $A$ represents the identity. Finally, the function is a group homomorphism since 
$det(A \oplus B) = det(A)det(B)$.
\end{proof}

\appendix
\section{Alternative formulation}\label{sec:variant}

In this appendix, we give an alternative definition the odd differential extension from Section \ref{sec:elemmodel}, which \emph{a prior} is defined via a larger generating set, but in fact turns out to be naturally equivalent to the differential extension $\hat{K}^{-1}(M)$ from Section \ref{sec:elemmodel}. This construction is closely related to the one given by Freed and Lott in \cite{FL}.

For any compact manifold $M$ one can always include the additional data of an even differential form, making the following definition:
\[
S(M) = \{ (g, X) | g \in Map (M, \U ), X \in \Om^{even}(M;\R) \}
\]
This set is a monoid under the operation $(g, X) + (h,Y) = (g \oplus h, X + Y)$ with identity given by $(1,0)$ where $1$ is the constant map to the identity $1\in U$.
We define an equivalence relation $\sim$ on $S(M)$ by declaring
\[
(g_0, X_0) \sim (g_1,X_1) 
\]
if and only if, there is a path $g_t$ from $g_0$ to $g_1$ such that: 
\[
 CS(g_t) \equiv X_1 - X_0  \, \, \textrm{mod exact forms}
\]
The relation $\sim$ defines an equivalence relation, and we denote the set of equivalence classes by $\Kalt(M) = S(M)/\sim$ \,, and we denote the equivalence class of a pair $(g,X)$ by $[g,X]$. 
It follows from Lemma \ref{lem:CSzero}, just as in the proof of Proposition \ref{prop:group}, that the equivalence classes form an abelian group under the sum operation described above. 

Furthermore, for a morphism $f:M'\to M$ of smooth manifolds, we define $f([g,X])=[g\circ f, f^*(X)]$ as the pullback of each component, so in this way $\Kalt$ defines a contravariant functor.

There is a commutative diagram of group homomorphisms 
\[
\xymatrix{
 & K^{-1}(M) \ar[rd]^{[Ch]} & \\
 \Kalt(M) \ar[ru]^{\Ialt} \ar [rd]^{\Ralt} & &H^{\textrm{odd}}(M;\R) \\
 & \Om^{\textrm{odd}}_{cl} (M;\R) \ar[ru]_{deRham} &
}
\]
which is natural in $M$, where  $K^{-1}(M) $ is set of homotopy classes of elements of $Map (M, \U)$, $\Ialt([g,X])$ is the homotopy class of $g$, and $\Ralt([g,X]) = Ch(g) + dX$.

For this construction of $\Kalt(M)$, an equivalence class $[g,X]$ has the additional data of an even differential form, so it is straightforward to define
\[
\aalt : \Om^{\textrm{even}} (M) / Im(d) \to \Kalt(M)
\]
 by 
 \[
 \aalt(X) = [1, X]
 \]
 where $1 \in Map (M, \U )$ is the constant map to $1 \in U$.

The map $\aalt$ is well defined since the Chern-Simons form associated to the constant path at $1 \in U$ is zero. It is straightforward to see that $\aalt$ is a group homomorphism. Furthermore, we have $\Ralt \circ \aalt = d$ since $\Ralt(\aalt([X])) = \Ralt([1,X]) = Ch(1) + dX = dX$.
Finally, we have:
\begin{prop} There is an exact sequence
\[
\xymatrix{
 K^0(M) \ar[r]^-{[Ch]} & \Om^{\textrm{even} } (M) / Im(d) \ar[r]^-{\aalt} &  \Kalt(M) \ar[r]^{\Ialt} & K^{-1}(M) \ar[r] & 0
}
\]
where $Ch : K^0(M) \to \Om^{\textrm{even} } (M) / Im(d)$ is the ordinary even Chern character map.
\end{prop}

\begin{proof} 
We first show that $Im( \aalt) = Ker (\Ialt)$. We have
\[
Ker(\Ialt) = \{ [g,X] | \, \textrm{there is a path $g_t$ such that $g(0) = g$ and $g(1) = 1$} \}
\]
while
\[
Im(\aalt) = \{ [1, X] | \, X \in \Om^{\textrm{even} } (M;\R) \}
\]
So, certainly $Im(\aalt) \subset Ker(\Ialt)$. Conversely, if $[g,X] \in Ker(\Ialt)$, and $g_t$ is a path such that $g(1) = 1$ and $g(0) = g$,
let $Y = CS(g_t) + X$. Then $(g,X) \sim (1,Y)$ since $Y - X = CS(g_t)$. This shows $[g,X] \in Im(\aalt)$ and therefore $Ker(\Ialt) \subset Im(\aalt)$.

Finally, we show that $Im([Ch]) = Ker(\aalt)$. First, note that
\[
Ker(\aalt) = \{ X   | \, X = CS(g_t) \,  \textrm{mod exact forms,  for some loop $g_t$ based at $1$} \},
\]
since $(1,X) \sim (1, 0)$ if and only if there is a path $g_t$ such that $g(0) = g(1) = 1$ and $X = CS(g_t)$ mod exact forms. So it suffices to show that the set of even Chern forms 
equals the set of differential forms of the form $CS(g_t)$ for some loop based at $1$ (modulo exact forms). This follows from Theorem \ref{thm:CSCh}.
\end{proof}

We now compare the differential extension $\hat{K}^{-1}$ from Section \ref{sec:elemmodel} with $\Kalt$ from this section.
 We use the natural notion of equivalence, \emph{i.e.} a natural isomorphism of functors $ \hat{K}^{-1} \to \Kalt$ that commutes with 
 additional data $(R,I,a)$ and $(S,J,b)$, see \cite[Definition 1.2]{BS3}.
 
\begin{prop}
The natural isomorphism $\Phi:\hat{K}^{-1}\to \Kalt$ given by $[g] \mapsto [g,0]$, where $g:M \to U$, is a natural equivalence from the differential extensions $(\hat{K}^{-1},R,I,a)$ to the differential extension $(\Kalt,S,J,b)$. 
\end{prop}
\begin{proof}
This map is easily seen to be well defined, injective and a natural homomorphism. 

The map $\Phi$ is surjective if and only if every $(g,X)$ is equivalent to some $(h,0)$, \emph{i.e.} for every map $g : M \to U$ and every even form $X$ there is a path $g_t$ such that $g_0=g$, $g_1=h$, and $CS(g_t) = X$ modulo exact. By Corollary \ref{cor:CSsurj} there is $k_t: M \times I \to U$ so that $k_0=1$ and $CS(k_t) = X$ modulo exact forms, so  for $g_t = g \oplus k_t$,  we have $g_0 =g$ and 
\[
CS(g_t) = CS(g) + CS(k_t) = CS(k_t) = X \quad \textrm{mod exact forms}.
\]
That is, $(g,X)$ is equivalent to $(g \oplus k_1, 0)$, so indeed the natural map is surjective, and 
$\Phi:\hat{K}^{-1}(M)\to \Kalt(M)$ is an isomorphism for each $M$.

It is immediate to check that  $\Ialt \circ \Phi=  I$ and  $\Ralt \circ \Phi= R $.
Finally, we  show that for $[X]\in \Om^{\textrm{even}}/Im(d)$,  $\Phi\circ a([X])$ equals $\aalt([X])= [1,X]$ in $\Kalt(M)$. Recall that $a([X])\in \hat{K}^{-1}(M)$ is defined as $[g_1]$ for any choice $g_t$ with $g_0=1$ and $CS(g_t)=[X]$ modulo $Im(Ch)$. By Corollary \ref{cor:CSsurj} we may choose $g_t$ so that $CS(g_t)=X\in \Om^{\textrm{even}}(M)$ modulo exact forms. 
Then, for this choice $g_t$, we have $\Phi\circ a([X])=\Phi([g_1])=[g_1,0] = [1,X] = \aalt([X])$. 
This completes the proof of the proposition.
\end{proof}

\end{document}